\documentclass{article}
\usepackage{ayv_paperstyle}
\usepackage{bbm}
\DeclareMathOperator*{\ArgMax}{Arg\,Max}
\DeclareMathOperator*{\Max}{Max}

\newcommand{\mean}[5]{\mathbb{E}^{#1,#2}_{#3,#4} \Bigg[ #5 \Bigg]}

\newcommand{\Qcal}{\mathcal{Q}}
\newcommand{\csol}{\mathbf{C}}
\newcommand{\gsol}{\mu_0\text{-}\mathbf{MR}}

\begin{document}
	\title{Planning problem for continuous-time finite state mean field game with compact action space}
	\author{{Yurii Averboukh\footnote{Higher School of Economics}\hspace{3pt}}, {Aleksei Volkov\footnote{Krasovskii Institute of Mathematics and Mechanics}}} 
	\maketitle
	
	
	\begin{abstract}
		The planning problem for the mean field game implies the one tries to transfer the system of infinitely many identical rational agents from the given distribution to the final one using the choice of the terminal payoff.
		It can be formulated as the mean field game system with the boundary condition only on the measure variable.
		In the paper, we consider the  continuous-time finite state mean field game assuming that the  space of actions for each player is compact.
		It is shown that the planning problem in this case may not admit a solution even if the final distribution is reachable from the initial one. 
		Further, we introduce the concept of generalized solution of the planning problem for the finite state mean field game based on the minimization of regret of the representative player. This minimal regret solution always exists.
		Additionally, the set of minimal regret solution is the closure of the set of classical solution of the planning problem provided that the latter is nonempty. 
		\msccode{49N80, 91A16, 
			60J27}
	\end{abstract}
	
	
	\section{Introduction}
	The mean field game theory examines  systems consisting of many similar players in the limit case when the number of players tends to infinity.
	This approach was first proposed by Lasry, Lions~\cite{Lasry_Lions_2006_I,Lasry_Lions_2006_II} and (independently) Huang, Caines, Malham\'{e}~\cite{Huang_Caines_Malhame_2007,Huang_Malhame_Caines_2006} in 2006.
	The key assumption of the  mean field game theory is that the players interact via some external media.
	This reduces the game-theoretical analysis of the system of infinitely many identical players to the couple of problems. The first problem studies the dynamics of the distribution of players, while the second one is concerned with the optimal behavior of the representative player.
	These  problems are linked  through  condition  that   all players use the strategy that is optimal for a representative player.
	Notice that the solutions of the aforementioned problems can be obtained through the mean field game system consisting of kinetic equation which  describes the motion of distribution of players and the Bellman equation providing the solution of the optimal control problem for the representative player.
	The original mean field game theory considers the boundary value problem when the kinetic equation is equipped by the condition on the initial distribution, while the Bellman equation has the terminal condition coming from the terminal payoff for a representative player.
	
	The mean field game theory comprises various types of the players' dynamics.
	The second-order mean field game implies that  the dynamics of each player is given by a stochastic differential equation.
	The case of deterministic evolution  is examined within the  first-order mean field game theory.
	In the paper, we assume that the dynamics of each player is governed by a continuous-type finite state Markov chain with transition rates depending on the action of the player and the distribution of all players.
	This assumption leads to the continuous-time finite state mean field game theory that finds applications in the study of social dynamics including botnet-security and models of corruption~\cite{Kolokoltsov_security,Kolokoltsov_Bensoussan,Kolokoltsov_Malafeev}.
	Notice that, in this case, the kinetic and Bellman equations are  ordinary differential equations.
	
	The study of continuous-time finite state mean field games started with papers~\cite{Basna_Hilbert_Kolokoltsov, Gomes_Mohr_Souza_finite_state}.
	These papers provide the existence result for the solution of finite state mean field game as well as the uniqueness condition and the construction of approximate equilibria for finite player Markov games those are based on a solution of the limiting mean field game (see also~\cite{Cecchin_Fischer_probabilistic_finite_state}).
	The master equation technique was developed for the finite state mean field games in~\cite{Bayraktar_Cohen_master}.
	Due to it, the games with common noise were studied in~\cite{Bayraktar_et_al_Wright_N_player,Bayraktar_Ceccin_et_al_common_noise}.
	However, these papers consider only the case when the action space is unbounded, and the Hamiltonian satisfies the coercivity condition. The finite state mean field games with compact action space was studied in \cite{Carmona_Wang}.
	Let us also mention paper~\cite{Averboukh_2021} where the finite state mean field game was reduced to a control problem.
	This gives the  description set of all solutions of the finite state mean field game.
	
	In the paper, we consider the planning problem for the mean field game.
	This problem was first proposed by Lions in his seminal lectures in Coll\`{e}ge de France~\cite{lions_lecture}.
	To describe the setting of the planning problem recall that the  classical problem for the mean field game implies the initial condition for the kinetic equation and terminal condition for the Bellman equation.
	The planning problem appears when we consider the mean field game system  with boundary conditions only for the kinetic equation, i.e., we fix the initial and final distributions of players.
	In the game-theoretical terms, this means that we wish to find a players' terminal payoff that provides the transfer of the distribution of all agent to the desired state under condition that each player chooses his/her control in the optimal way.
	
	The planning problem for the mean field games with dynamics of each player given by either ordinary or stochastic differential equations was examined in papers~\cite{Bakaryan_Ferrira_Gomes_2,Bakaryan_Ferrira_Gomes, Graber_et_al_2019,var_approach_mfg,Porrtte}.
	Those papers provide the existence result in the case of Hamiltonian depending on the adjoint variable in the superlinear way.
	Notice that paper~\cite{Porrtte} uses the penalty approach, whilst papers~\cite{Uniform_estimates,Bakaryan_Ferrira_Gomes_2,Bakaryan_Ferrira_Gomes,Graber_et_al_2019,var_approach_mfg} consider potential mean field games and reduce the planning problem to an optimization problem.
	This also allows to examine the regularity of solutions of the planning problem for the potential mean field game.
	The planning problem for the finite state mean field game was considered in~\cite{Bertucci_2021}. In that paper, the existence result for master equation was derived for the case of unbounded control space under monotonicity assumptions.
	
	In the paper, we examine the planning problem for the continuous-time finite state mean field game assuming that (i) the action space is compact, (ii) the  transition rates and the running payoffs are continuous, (iii) the transition rates are Lipschitz continuous w.r.t. the measure variable.
	We show that this problem may not admit a solution even in the case when the final distribution is reachable from the initial one.
	Further, we introduce and study  generalized solutions. We propose the  concept of a minimal regret solution. It is defined as a cluster point of  solutions of the auxiliary control problems, where  an external decision maker chooses the players' strategy and the terminal payoff from some ball trying to minimize an averaged regret. 
	Moreover, we claim that the chosen strategy stirs  the distribution of players to the given final state.
	The limit is considered in the case when the size of balls tends to infinity.
	We prove that  the set of minimal regret solutions is nonempty and closed in the topology of narrow convergence, while, if there exists at least one classical solution of the planning problem for the finite state mean field game, then the closure of the set of classical solutions coincides with the set of minimal regret solutions. Finally, we study the uniqueness of the solution to the planning problem for the particular case where the transition rates do not depend on the players' distribution, running payoff satisfies the Lasry-Lions monotonicity condition and there exists at least one classical solution. 
	
	The paper is organized as follows. In Section~\ref{sect:statement}, we recall the concepts of finite state mean field games and planning problem.
	The randomized strategies used for the analysis of the motion of the representative player are introduced in Section~\ref{sect:randomized_control}.
	The next section contains an example providing the non-existence of the solution to the planning problem for the finite state mean field game.
	To introduce the minimal regret solution, we reformulate the planning problem for finite state mean field game as a control problem with mixed constraints in Section~\ref{sect:equivalent}.
	The   minimal regret solution is defined in Section~\ref{sect:regularization}.
	Additionally, this section provides the existence result for minimal regret solution as well as the link between the classical and minimal regret solutions to the planning problem for the finite state mean field games. Finally,  Section \ref{sect:uniqueness} is concerned with the uniqueness of the solution to the planning problem under some assumptions those include the monotonicity condition.
	
	\section{Problem statement}\label{sect:statement}
	\subsection{Mean Field Game}
	The  continuous-time finite state mean field game theory studies the case where the dynamics of each player is determined by a continuous-time finite state Markov chain.
	Without loss of generality, we assume that the state space for each player is  $\{ 1, \ldots, d \}$.
	
	Additionally, we fix the space of elementary events $\Omega$ and put it equal to the space of c\`adl\`ag fucntions from $[0,T]$ to $\{1,\ldots,d\}$, i.e.,
	\[\Omega=D([0,T];\{1,\ldots,d\}).\]  Further, let $\mathcal{F}$ be equal to the Borel $\sigma$-algebra on $\Omega$. Recall that it is generated by the cylinders \cite[Section~12]{Billingsley}.
	
	Additionally, we fix a stochastic process, letting
	$X(t,\omega)\triangleq \omega(t)$. This means that the choice of the control in  the Markov chain will directly affect the probabilities on $\mathcal{F}$. If $\mathbb{P}$ is a probability on $\mathcal{F}$, then one can introduce the instantaneous distribution $m(t) = (m_1(t), \ldots,  m_d(t))$ on $\{1,\ldots,d\}$ by the following rule:
	\begin{equation}\label{intro:distribution}
		m_i(t) = \mathbb{P}(\{\omega(t) = i\}).
	\end{equation}
	
	Below, we assume that $m(t)$ is a row-vector. Notice that $m(t)$ lies in the $d$-dimensional simplex $\Sigma^d$:
	\[
	\Sigma^d=\{m=(m_1,\ldots,m_d):m_1,\ldots,m_d\geq 0,\ \ m_1+\ldots,m_d=1\}.
	\] If $i\in \{1,\ldots,d\}$, then denote by $\mathbbm{1}_{k}$ a distribution on $\{1,\ldots,d\}$ concentrated in $k$, i.e., $\mathbbm{1}_k=(\mathbbm{1}_{k,1},\ldots,\mathbbm{1}_{k,d})$ where
	$\mathbbm{1}_{k,j}=1$ if $k=j$ and $0$ otherwise. 
	
	The mean field game setting implies that the distribution of the players at time $t$ is equal to $m(t)$, while the tansition rates for each player depends on the current distribution of players, i.e.,  the dynamics of the each player is determined  by the Markov chain with the  Kolmogorov matrix
	\[
	Q(t,m(t),u(t)) = \begin{pmatrix} Q_{ij}(t,m(t),u(t)) \end{pmatrix}_{i,j=1}^d.
	\]
	Here $u(t)$ stands for the instantaneous player's control.
	We assume that $u(t)\in U$.
	The property that $Q(t,m,u)$ is a Kolmogorov matrix means that, for each $t\in [0,T]$, $m\in\rd$, $u\in U$, $i\in \{1,\ldots,d\}$,
	\[Q_{i,j}(t,m,u)\geq 0,\text{ when }j\neq i\]
	and
	\[
	\sum_{j=1}^d Q_{ij}(t,m,u) = 0
	\]
	for every $t \in [0,T]$, $m \in \Sigma^d$, $u \in U$.
	The distribution of players $m(\cdot)$ defined by the rule \eqref{intro:distribution} obeys the nonlinear Kolmogorov equation (see \cite[Theorem 3.11]{Kol_book})
	\[
	\frac{d}{dt}{m}(t) = m(t)Q(t,m(t),u(t)), \quad m(0) = m_0.
	\]
	
	Here $t$ stands for time, $m(t)$ denotes the distribution of all player at time $t$, whilst $u(t)\in U$ is the instantaneous control acting upon the representative player.
	The original setting of the finite state mean field game theory (see~\cite{Basna_Hilbert_Kolokoltsov,Gomes_Mohr_Souza_finite_state}) implies that each player tries to maximize the quantity
	\begin{equation}\label{intro:expect:reward}
		\mathbb{E}\Bigg[\sigma_{X(T)}(m(T)) + \int_{0}^T g_{X(t)}(t,m(t),u(t)) dt\Bigg].
	\end{equation}
	Hereinafter, $\mathbb{E}$ stands for the expectation, whilst $\sigma:\Sigma^d\rightarrow\mathbb{R}^d$, $g:[0,T]\times \Sigma^d\times U\rightarrow \rd$ describe terminal and running payoffs respectively.
	
	We adopt the feedback approach. This means that each player player uses the control $u_i(t)$ whenever he/she occupies the state $i$ at the time $t$.
	
	Thus, without loss of generality, one can assume that the space of actions now is equal to $U^d$ and the instantaneous actions   is a $d$-dimensional vector
	\[
	u(t) = (
	u_1(t),
	\ldots,
	u_d(t)).
	\]
	In this case, the  Kolmogorov matrix takes the form
	\[
	\Qcal(t,m,u) = \begin{pmatrix} Q_{ij}(t,m,u_i) \end{pmatrix}_{i,j=1}^d.
	\]
	Here $t\in [0,T]$, $m\in\Sigma^d$, $u\in U^d$. Furthermore, for $m\in\Sigma^d$, $u=(u_i)_{i=1}^d\in U^d$, we put
	\[
	g(t,m,u)=\left(\begin{array}{c}
		g_1(t,m,u_1)\\
		\vdots\\
		g_d(t,m,d)
	\end{array}\right).
	\]
	
	Further, assume that we are given with a flow of distributions $m(\cdot)$, feedback strategy $u(\cdot)$, and some element of the simplex $\mu_0=(\mu_{0,1},\ldots,\mu_{0,d})$. One can consider the Markov chain with the Kolmogorov matrix $\Qcal(t,m(\cdot),u)$, an initial time $s$ and an initial distribution $\mu_0$. We denote the corresponding probability on $\mathcal{F}$ by $\mathbb{P}^{s,\mu_0}_{u(\cdot),m(\cdot)}$. Additionally, let $\mathbb{E}^{s,\mu_0}_{u(\cdot),m(\cdot)}$ stands for the corresponding expectation. 
	One can interpret the probability $\mathbb{P}^{s,\mu_0}_{u(\cdot),m(\cdot)}$ as a description of the evolution of a fictitious player whose state at the initial time $s$ is distributed according to $\mu_0$ and who uses the strategy $u(\cdot)$ while the dynamics of the whole mass of players is given by $m(\cdot)$. Indeed, if we denote $\mu_i(t)=\mathbb{P}^{s,\mu_0}_{u(\cdot),m(\cdot)}\{X(t)=i\}$, then $\mu(\cdot)=(\mu_1(\cdot),\ldots,\mu_d(\cdot))$ obeys the Kolmogorov equation
	\begin{equation}\label{eq:Kolmogorov_linear}
		\frac{d}{dt}\mu(t)=\mu(t)\Qcal(t,m(t),u(t)),\ \ \mu(s)=\mu_0.
	\end{equation} 
	
	Further, if $\phi$ is a column-vector, then, for $t\in [s,T]$,
	\begin{equation}\label{equality:expectation}
		\mathbb{E}^{s,\mu_0}_{u(\cdot),m(\cdot)}\phi_{X(t)}=\mu(t)\phi.
	\end{equation}

	Notice that, a nonlinear Markov chain corresponds to the equalities
	\begin{equation}\label{equality:nonlinear}m_i(t)=\mathbb{P}^{0,m_0}_{u(\cdot),m(\cdot)}\{X(t)=i\},\ \ i=1,\ldots,d,\ \  m(0)=m_0.\end{equation} Due to \cite[Theorem 3.11]{Kol_book}, equalities \eqref{equality:nonlinear} can be rewritten as a  nonlinear Kolmogorov equation on $m(\cdot)$: 
	\begin{equation}\label{eq:Kolmogorov_feedback}
		\frac{d}{dt}m(t) = m(t)\Qcal(t,m(t),u(t)),\ \ m(0)=m_0.
	\end{equation}

	The concept of the finite state  mean field games implies that we wish to find a feedback strategy $u^*$ such that there exists a flow of probabilities $m^*(\cdot)$ satisfying \eqref{equality:nonlinear} (or, equivalently \eqref{eq:Kolmogorov_feedback}), while, for each $k\in \{1,\ldots,d\}$ and every feedback strategy $u$, one has
	\begin{equation}\label{ineq:payoff_i}
		\begin{split}\mean{s}{\mathbbm{1}_k}{u^*(\cdot)}{m^*(\cdot)}{&\sigma_{X(T)}(m^*(T)) + \int_{0}^T g_{X(t)}(t,m^*(t),u^*_{X(t)}(t)) dt}\\&\geq \mean{s}{\mathbbm{1}_k}{u(\cdot)}{m^*(\cdot)}{\sigma_{X(T)}(m^*(T)) + \int_{0}^T g_{X(t)}(t,m^*(t),u_{X(t)}(t)) dt}.
	\end{split}\end{equation}  The latter inequality can be rewritten using the concept of the value function that, in its own turn, leads to the  mean field game system.
	
	For each $s\in [0,T]$, we regard $\varphi(s)$ as a column-vector:
	\[
	\varphi(s) = \begin{pmatrix}
		\varphi_1(s)\\
		\vdots\\
		\varphi_d(s)
	\end{pmatrix}.
	\]
	The quantity $\varphi_k(s)$ is the maximal outcome of the representative player who starts at time~$s$ from the state $k$:
	\begin{equation} \label{intro:varphi}
		\begin{split}
			\varphi_k&(s) \\&= \sup\left\{ \mean{s}{\mathbbm{1}_k}{u(\cdot)}{m^*}{\sigma_{X(T)}(m^*(T))+ \int_{s}^T g_{X(t)}(t,m^*(t),u_{X(t)}(t)) dt}\right\}.\end{split}
	\end{equation}
	In~(\ref{intro:varphi}), the supremum is taken over the set of all feedback strategies $u$.
	
	These system of $d$ equalities can be reduced to only one equality as follows. Let $\mu_0=(\mu_{0,1},\ldots,\mu_{0,d})\in \Sigma^d$ be such that $\mu_{0,k}$ are strictly positive. Then, since we assume that there exists an universal optimal strategy $u^*(\cdot)$, system~\eqref{intro:varphi} is equivalent to the equality
	\[\begin{split}	\mu_0&\varphi(s) \\= &\sup\left\{\sum_{k=1}^d\mu_{0,k} \mean{s}{\mathbbm{1}_k}{u(\cdot)}{m^*}{\sigma_{X(T)}(m^*(T))+ \int_{s}^T g_{X(t)}(t,m^*(t),u_{X(t)}(t)) dt}\right\}.\end{split}\] Here, as above, the supremum is taken over the set of all feedback strategies. Further, notice that, for each $A\in\mathcal{F}$, 
	\[\sum_{k=1}^d\mu_{0,k}\mathbb{P}^{s,\mathbbm{1}}_{u(\cdot),m(\cdot)}(A)=\mathbb{P}^{s,\mu_0}_{u(\cdot),m(\cdot)}(A).\] This, \eqref{eq:Kolmogorov_linear} and the assumption that $\mu_{0,k}<0$ give that
	\begin{equation}\label{equality:varphi_mean}\mu_0\varphi(s) = \sup\left\{\mu_0\sigma_{X(T)}(m^*(T))+ \int_{s}^T\mu(t) g_{X(t)}(t,m^*(t),u(t)) dt\right\}.\end{equation} Here the supremum is taken over the set of pairs $(\mu(\cdot),u(\cdot))$ satisfying \eqref{eq:Kolmogorov_linear}. 
	
	The value function obeys the Bellman equation.
	To give a concise form of this equation, let us introduce the notion coordinate-wise maximum that will be  widely used below.
	If $Z$ is a set, $f_i:Z\rightarrow \mathbb{R}$, then define the function $f:Z^d\rightarrow \rd$ by the rule: for $z=(z_1,\ldots,z_d)$, where $z_i \in Z$,
	\begin{equation}\label{intro:f_vector}
		f(z)=(f_1(z_1),\ldots,f_d(z_d)).
	\end{equation}
	Now, let us introduce the coordinate-wise maximum by the rule:
	\begin{equation}\label{intro:coordinate_wise_max}
		\Max\limits_{z \in Z^d} f(z) = \begin{pmatrix}
			\max\limits_{z_1 \in Z} f_1(z_1)\\
			\vdots\\
			\max\limits_{z_d \in Z} f_d(z_d)\\
		\end{pmatrix}.
	\end{equation}
	The inverse operation is denoted by $\ArgMax$, i.e.,
	\begin{equation}\label{intro:coordinate_wise_ArgMax}
		\ArgMax\limits_{z \in Z^d} f(z) = \{ w \in Z^d : f(w) = \Max\limits_{z \in Z^d} f(z) \}.
	\end{equation}
	
	The main object for the Bellman equation is the Hamiltonian  defined by the following rule for each $t\in [0,T]$, $m\in\Sigma^d$, $\phi\in\rd$:
	\begin{equation} \label{intro:hamiltonian}
		H(t,m,\phi) = \Max\limits_{u \in U^d} \{ \Qcal(t,m,u) \phi + g(t,m,u) \},
	\end{equation}
	
	The Bellman equation takes the form
	\begin{equation}\label{eq:Bellman_usual}
		\frac{d}{dt}{\varphi}(t) = - H(t,m^*(t),\varphi(t)), \quad \varphi(T) = \sigma(m(T)).
	\end{equation} Inequality \eqref{ineq:payoff_i} now means that 
	\[u^*(t) \in \ArgMax\limits_{u \in U^d} \left\{ \Qcal(t,m^*(t),u) \varphi(t) + g(t,m^*(t),u) \right\},\] while $\varphi(\cdot)$ satisfies the Bellman equation.
	
	Thus, original problem for the finite state mean field game can be regarded as the following boundary problem
	\begin{equation}\label{eq:mfg_system}
		\begin{split}
			&\frac{d}{dt}{m}^*(t) = m(t)\Qcal(t,m^*(t),u^*(t)),\\
			&\frac{d}{dt}{\varphi}(t) = - H(t,m^*(t),\varphi(t)),\\
			&u^*(t) \in \ArgMax\limits_{u \in U^d} \left\{ \Qcal(t,m^*(t),u) \varphi(t) + g(t,m^*(t),u) \right\},\\
			&m(0) = m_0,\ \ \varphi(T) = \sigma(m(T)).
		\end{split}
	\end{equation}
	
	\subsection{Planning problem}
	One may follow the Lions' lectures given in Coll\`{e}ge de France \cite{lions_lecture} and define the planning problem replacing the boundary condition $m(0) = m_0,$ $\varphi(T) = \sigma(m(T))$ with the conditions only on distribution:
	\[
	m(0) = m_0,\ \  m(T) = m_T,
	\]
	where $m_T$ is a desired destination, i.e., the planning problem is determined by the following system:
	\[
	\begin{split}
		&\frac{d}{dt}{m}(t) = m(t)\Qcal(t,m(t),u^*(t)), \\
		&\frac{d}{dt}{\varphi}(t) = - H(t,m(t),\varphi(t)), \\
		&u^*(t) \in \ArgMax\limits_{u \in U^d} \left\{ \Qcal(t,m(t),u) \varphi(t) + g(t,m(t),u) \right\}\\
		&m(0) = m_0,\ \ m(T) = m_T.
	\end{split}
	\]
	
	Let us give the game-theoretical interpretation of the planing problem. It can be regarded as a two-level game. On the first level, an external decision maker chooses a terminal payoff $\sigma$. The second level game is a noncooperative Markov game with infinitely many identical player. Here, the dynamics of each player is determined by the Kolmogorov matrix $Q$, whilst his/her payoff is equal to~(\ref{intro:expect:reward}).
	
	We will study the planning problem under the following condition on the action space $U$, the Kolmogorov matrix $Q$ and the integral payoff $g$:
	\begin{enumerate}[label=(C\arabic*), series=listCond]
		\item\label{cond:compact} $U$ is metric compact,
		\item\label{cond:continuity} $Q_{ij},  g$ are continuous,
		\item\label{cond:Lipschitz} $Q_{ij}$ are Lipschitz continuous w.r.t. $m$.
	\end{enumerate}
	
	\section{Randomized control}\label{sect:randomized_control}
	Above we considered only the feedback strategies. This choice, in particular, implies that the all players who  occupies at time $t$ the  state $i$ should use the same control $u(t,i)$. However, this assumption is too restrictive in the case when the Hamiltonian is not strictly convex.
	Moreover, the set of feedback measurable control is neither convex nor closed. Thus, it is reasonable to use the randomized (relaxed) feedback strategies. 
	
	In the following, $\mathcal{P}(U)$ stands for the set of probabilities on $U$. 
	
	\begin{definition}
		A \textit{randomized feedback strategy} is a sequence $\nu(t) =(\nu_i(t))_{i=1}^d $ such that, for each $i$,
		\begin{itemize}
			\item $\nu_i(t)$ is a probability on $U$ for each $t \in [0,T]$,
			\item the mapping $t\mapsto \nu_i(t,\cdot)$ is weakly measurable, i.e., the mapping $t \mapsto \int_U \zeta(u) \nu_i(t,du)$ is measurable for every continuous function $\zeta : U \rightarrow \mathbb{R}$.
		\end{itemize}
		Here, to simplify notation, we write $\nu_i(t,du)$ instead of $\nu_i(t)(du)$.
	\end{definition}
	Below, we denote the set of all randomized strategies by $\mathcal{U}$.
	
	In the case of randomized strategies the transition rates (still denoted by $Q_{ij}$) for $t\in [0,T]$, $m\in\Sigma^d$, $\nu_i\in \mathcal{P}(U)$ are equal to
	\[
	Q_{ij}(t,m,\nu_i)\triangleq \int_U Q_{ij}(t,m,u)\nu_i(t,du),
	\]
	whilst the running payoff is given by
	\[
	g_i(t,m,\nu_i)\triangleq\int_U g_i(t,m,u)\nu_i(t,du).
	\]
	Thus, if, as above, $t\in [0,T]$, $m\in\Sigma^d$, while $\nu=(\nu_i)_{i=1}^d\in \mathcal{P}(U)$ is a instantaneous randomized control, then the Kolmogorov matrix in the framework of randomized strategies is 
	\[
	\Qcal(t,m,\nu) = \begin{pmatrix}  Q_{ij}(t,m,\nu_i) \end{pmatrix}_{i,j=1}^d.
	\] 
	If $s$ is an initial time, $\mu_0$ is a state at time $s$, $\nu(\cdot)=(\nu_i(\cdot))_{i=1}^d$ is a randomized feedback strategy, while $m(\cdot)$ is a flow of probabilities, we denote the probability in the Markov chain generated by the Kolmogorov matrix $\Qcal(t,m(t),\nu(t))$ by $\mathbb{P}^{s,\mu_0}_{\nu(\cdot),m(\cdot)}$. We also use the designation  $\mathbb{E}^{s,\mu_0}_{\nu(\cdot),m(\cdot)}$ for the corresponding expectation. 
	
	With some abuse of notation, we use the letter $g$ to describe the running payoff for the case of randomized strategies, i.e., we put
	\[
	g(t,m,\nu) = \begin{pmatrix}
		g_1(t,m,\nu_1)\\
		\vdots\\
		g_d(t,m,\nu_d)
	\end{pmatrix}
	\]
	
	Notice that the use of randomized feedback strategies leads to the replacement of $u(\cdot)$ with $\nu(\cdot)$ in \eqref{equality:nonlinear}--\eqref{intro:varphi}.
	
	Therefore, the definition of the solution of the planning problem for the case of randomized strategies takes the following form.
	
	\begin{definition}\label{def:planning_classical}
		Let  $m_0$ be an initial distribution, and let $m_T$ be a terminal distribution.
		We say that $\nu^*(\cdot)$ is a \textit{solution of  the planning problem} provided that
		\begin{enumerate}[label=(\roman*)]
			\item\label{def:planning_classic:distribution} there exists $m^*(\cdot)$ satisfying the Kolmogorov equation and the boundary conditions
			\[
			\frac{d}{dt}{m}^*(t) = m^*(t) \Qcal(t,m^*(t),\nu^*(t)), \quad m^*(0) = m_0,\, m^*(T) = m_T,
			\]
			\item\label{def:planning_classic:value} there exists $\varphi^*(\cdot)$ satisfying the Bellman equation
			\[
			\frac{d}{dt}{\varphi}^*(t) = - H(t,m^*(t),\varphi^*(t)),
			\]
			\item\label{def:planning_classic:optimal} $\nu^*(\cdot)$ is the optimal randomized feedback control, i.e.,
			\[
			\nu^*(t) \in \ArgMax\limits_{\nu \in \mathcal{P}(U)^d} \left\{ \Qcal(t,m^*(t),\nu) \varphi^*(t) + g(t,m^*(t),\nu) \right\}.
			\]
		\end{enumerate}
		If $\nu^*(\cdot)$ is a solution of the planning problem and $m^*(\cdot)$, $\varphi^*(\cdot)$ satisfy conditions of Definition \ref{def:planning_classical}, then we say that  the distribution $m^*(\cdot)$ and the value function $\varphi^*(\cdot))$ \textit{corresponds} to the solution $\nu^*(\cdot)$. 
	\end{definition}
	Below, we denote the  set of solutions of the planning problem by $\csol$.
	Notice that it is reasonable to restrict our attention to the case when $m_T$ is reachable from $m_0$.
	Indeed, in the opposite case, the planning problem does not admit a solution.
	We formulate the reachability condition as follows.
	\begin{enumerate}[label=(C\arabic*), resume*=listCond]
		\item\label{cond:reacability} There exists a feedback strategy $\nu(\cdot)$ and a function $m(\cdot)$ such that
		\[
		\frac{d}{dt}{m}(t) = m(t) \Qcal(t,m(t),\nu(t)), \ m(0) = m_0,\ \ m(T)=m_T.
		\]
	\end{enumerate}
	
	\section{On nonexistence of the solution of the planning problem} \label{section:example}
	The aim of this section is to show that, even if reachability assumption \ref{cond:reacability} holds true, the solution of the planning problem may not exist.
	To this end, let us consider the following example of the planning problems with the parameters:
	\[
	d = 3, \ \  T = 1, \ \ u(t) \in [0,1],
	\]
	\[
	m_0 = \begin{pmatrix} 1, & 0, & 0 \end{pmatrix}, \ 
	m_T = \begin{pmatrix}  1-e^{-1/3}, & e^{-1/3}, & 0 \end{pmatrix}, \]
	\[
	Q(t,m,u)=\begin{pmatrix}
		-u & u & 0\\
		0 & -\rho(t) & \rho(t)\\
		0 & 0 & 0
	\end{pmatrix}, \ \ g(t,m,u) = \begin{pmatrix}
		-u^2\\
		0\\
		0
	\end{pmatrix},
	\]
	where
	\[
	\rho(t) = \begin{cases}
		1, & 0 \leq t < \frac13\\
		-3t + 2, & \frac13 \leq t < \frac23,\\
		0, & \frac23 \leq t \leq 1.
	\end{cases}
	\]
	
	Notice that, since, for each $i=1,2,3$, the function 
	\[
	u\mapsto \sum_{j=1}^3 Q_{i,j}(t,u)\varphi_j+g_i(u)
	\]
	is concave for each $\phi=(\phi_1,\phi_2,\phi_3)^T\in\mathbb{R}^3$, we do not need the randomization and can examine the problem using measurable strategies.
	
	The unique control stirring the distribution of agents from $m_0$ to $m_T$ is
	\[
	\tilde{u}(t) = \begin{cases}
		0, & 0 \leq t < \frac23,\\
		1, & \frac23 \leq t \leq 1.
	\end{cases}
	\]
	
	Indeed, since $\rho(t)>0$ on $[0,\frac23]$, the only way to assure the equality ${m_T}_3=0$ is to put $u(t)=0$ on $[0,\frac23]$. 
	Therefore, for $t\in [\frac23,1]$, one has that
	\[
	m_1(t)=\exp\left(-\int_{2/3}^t u(\tau)d\tau\right),\ \ m_2(t)=1-m_1(t),\ \ m_3(t)\equiv 0,
	\]
	Thus, the system achieves the final distribution $m_T$, only if we keep  the control $u(t)\equiv 1$ on $[\frac23,1]$.
	
	To show that $\tilde{u}(\cdot)$ does not solve the planning problem, let us consider the  Bellman equation. Recall that now it is a system of three ODEs. The control is determined by the first one which  takes the form
	\begin{equation*}\label{eq:example_Bellman_1}
		\frac{d}{dt}{\varphi}_1(t) = - \max\limits_{u\in [0,1]} \left\{ u (\varphi_2(t) - \varphi_1(t))  - \frac{u^2}2 \right\}.
	\end{equation*}
	Hence, the  strategy that satisfies condition \ref{def:planning_classic:optimal} of Definition \ref{def:planning_classical} in this case is
	\[
	u^*(t) = \begin{cases}
		0, & \varphi_2(t) - \varphi_1(t) \leq 0,\\
		\varphi_2(t) - \varphi_1(t), & 0 \leq \varphi_2(t) - \varphi_1(t) \leq 1,\\
		1, & 1 \leq \varphi_2(t) - \varphi_1(t).
	\end{cases}
	\]
	
	Notice that the  functions $\varphi_1(\cdot)$, $\varphi_2(\cdot)$ solve the Bellman equation, and, thus are  continuous. This yields the continuity of the   control $u^*$. At the same time, the control $\tilde{u}(\cdot)$ is discontinuous.
	Hence, $\tilde{u} \neq u^*$ and the considered planning problem does not admit a solution.
	
	Notice that the presented example is quite different from the previous existence result for the planning problem obtained in \cite{var_approach_mfg,Bertucci_2021,Porrtte,Bakaryan_Ferrira_Gomes,Graber_et_al_2019}. Probably, this is due to the fact that the aforementioned papers either deal with the case of unbounded action space ~\cite{var_approach_mfg,Porrtte,Bakaryan_Ferrira_Gomes,Graber_et_al_2019} or include the strict monotonicity condition~\cite{Bertucci_2021}.
	
	\section{Equivalent control problem}\label{sect:equivalent}
	In this section, we reformulate the planning problem for the finite state mean field game as a control problem. This equivalent form will serve as a hint for the definition of  generalized solution to the planning problem introduced in the next section. Additionally, we will use the reformulation result to obtain the link between the classical and generalized solutions of the planning problem (see Theorem \ref{th:link_minimal_regret} below).   Below we assume that conditions (C1)--(C4) are in force.
	
	To explain the way how we derive the equivalent form of Definition \ref{def:planning_classical}, we notice that conditions \ref{def:planning_classic:value} and \ref{def:planning_classic:optimal} of this definition imply that, for each initial time $s$ and each initial state $k$, 
	\begin{equation}\label{intro:varphi_nu}
		\begin{split}	
			\varphi_k(s)&=\mean{s}{\mathbbm{1}_k}{\nu^*(\cdot)}{m(\cdot)}{\varphi_{X(T)}(T)+ \int_{s}^Tg_{X(t)}(t,m(t),\nu^*(t))dt}\\&= \max\Bigg\{\mean{s}{\mathbbm{1}_k}{\nu(\cdot)}{m(\cdot)}{\varphi_{X(T)}(T)+ \int_{s}^Tg_{X(t)}(t,m(t),\nu(t))dt}: \nu(\cdot)\in\mathcal{U}\Bigg\}.\end{split}\end{equation}
	Multiplying the $k$-th equality \eqref{intro:varphi_nu} by $\mu_{0,k}$ where \[\mu_{0,k}>0\text{ and } \mu_{0,1}+\ldots+\mu_{0,d}=1,\] we conclude that system of equality \eqref{intro:varphi_nu} is equivalent to the following:
	\begin{equation}\label{intro:varphi_nu_mu_0}
		\begin{split}	
			\mu_0\varphi(s)&=\sum_{k=1}^d\mu_{0,k}\mean{s}{\mathbbm{1}_k}{\nu^*(\cdot)}{m(\cdot)}{ \varphi_{X(T)}(T)+ \int_{s}^Tg_{X(t)}(t,m(t),\nu^*(t))dt}\\&= \sum_{k=1}^d\mu_{0,k}\max\Bigg\{\mean{s}{\mathbbm{1}_k}{\nu(\cdot)}{m(\cdot)}{\varphi_{X(T)}(T)\\&\hspace{90pt}+ \int_{s}^Tg_{X(t)}(t,m(t),\nu(t))dt}: \nu(\cdot)\in\mathcal{U}\Bigg\}.\end{split}\end{equation}	
	Here $\mu_0=(\mu_{0,1},\ldots,\mu_{0,d})$.
	
	Now let us consider the probability  $\mathbb{P}^{s,\mu_0}_{\nu(\cdot),m(\cdot)}$. Recall that, for each $A\in\mathcal{F}$,
	\[\mathbb{P}^{s,\mu_0}_{\nu(\cdot),m(\cdot)}(A)=\sum_{k=1}^d\mu_{0,i}\mathbb{P}^{s,\mathbbm{1}_k}_{\nu(\cdot),m(\cdot)}(A).\] Therefore, since we assume that $\mu_{0,k}>0$, \eqref{intro:varphi_nu_mu_0} takes the form:
	\begin{equation}\label{intro:varphi_nu_mu_0_fin}
		\begin{split}	
			\mu_0&\varphi(s)\\=&\mean{s}{\mu_0}{\nu^*(\cdot)}{m(\cdot)}{ \varphi_{X(T)}(T)+ \int_{s}^Tg_{X(t)}(t,m(t),\nu^*(t))dt}\\= &\max\Bigg\{\mean{s}{\mu_0}{\nu(\cdot)}{m(\cdot)}{\varphi_{X(T)}(T)+ \int_{s}^Tg_{X(t)}(t,m(t),\nu(t))dt}: \nu(\cdot)\in\mathcal{U}\Bigg\}.\end{split}\end{equation}
	Further, let $\mu(\cdot)=(\mu_1(\cdot),\ldots,\mu_d(\cdot))$ be defined by the rule:   $\mu_k(t)=\mathbbm{P}^{s,\mu_0}_{\nu(\cdot),m(\cdot)}(\{X(t)=k\})$ and satisfies the Kolmogorov equation for the following Kolmogorov equation that is a relaxed version of \eqref{eq:Kolmogorov_linear}:
	\begin{equation}\label{eq:Kolmogorov_linear_nu}
		\frac{d}{dt}\mu(t)=\mu(t)\mathcal{Q}(t,m(t),\nu(t)),\  \ \mu(s)=\mu_0.
	\end{equation} We have that, for every vector $\phi\in\rd$,  $\mathbbm{E}^{s,\mu_0}_{\nu(\cdot),m(\cdot)}\phi_{X(t)}=\mu(t)\phi$. Therefore, we compute the right-hand side in \eqref{intro:varphi_nu_mu_0_fin} and have the equality:
	\begin{equation}\label{equality:varphi_mean_mu}
		\mu_0\varphi(s)=\max\Bigg\{\mu(T)\varphi(T)+\int_s^T\mu(t)g(t,m(t),\nu(t))dt:\nu\in\mathcal{U}\Bigg\}.
	\end{equation} Below (see Theorem \ref{th:equivalent}), we show that the solution of the planning problem can be determined by this equality only for $s=0$.
	
	Notice that the probability $\mathbbm{P}^{s,\mu_0}_{\nu(\cdot),m(\cdot)}$ describes the evolution of the fictitious player whose states at time $s$ are distributed according to $\mu_0$ and who plays with the randomized feedback strategy $\nu(\cdot)$ while the distribution of the whole mass of players at each time $t$ is $m(t)$. Therefore, the quantity
	\[\mu(T)\sigma(m(T))+\int_s^T\mu(t)g(t,m(t),\nu(t))dt\] is a reward of this fictitious player, while one can regard the number
	\[\mu_0\varphi(s)-\mu(T)\sigma(m(T))-\int_s^T\mu(t)g(t,m(t),\nu(t))dt\] as his/her regret on time interval $[s,T]$. 
	
	\begin{theorem}\label{th:equivalent} Let $\mu_0$ be a fixed element of $\Sigma^d$ with nonzero coordinates.
		The function $\nu^*(\cdot)$ is the solution of the planning problem if and only if  there exists a 4-tuple $(m^*(\cdot),\varphi^*(\cdot),\mu^*(\cdot),\nu^*(\cdot))$ providing the solution of optimal control problem the following optimal control problem with the mixed constraints:
		\begin{equation} \label{functional:J_control}
			\begin{split}
				\text{minimize } J(m(\cdot),\varphi&(\cdot),\mu(\cdot),\nu(\cdot)) =\\
				\mu_0 \varphi(0) - &\mu(T) \varphi(T) - \int_{0}^T \mu(t) g(t,m(t),\nu(t)) dt
			\end{split}
		\end{equation}
		subject to
		\begin{equation} \label{sys:control}
			\begin{split}
				\frac{d}{dt}{m}(t) &= m(t) \Qcal(t,m(t),\nu(t)), \ m(0) = m_0,\ \ m(T)=m_T\\
				\frac{d}{dt}{\varphi}(t) &= - H(t,m(t),\varphi(t))\\
				\frac{d}{dt}{\mu}(t) &= \mu(t) \Qcal(t,m(t),\nu(t)), \ \mu(0) = \mu_0.\\
			\end{split}
		\end{equation} with
		\[
		J(m^*(\cdot),\varphi^*(\cdot),\mu^*(\cdot),\nu^*(\cdot)) = 0.
		\]
	\end{theorem}
	
	\begin{remark}\label{rem:control_regret}
		In \eqref{functional:J_control}, \eqref{sys:control}, the control parameters are the terminal payoff $\varphi(T)\in\rd$ and the randomized feedback strategy~$\nu(\cdot)$.
	\end{remark}
	\begin{remark}\label{rem:nonzero} The assumption that $\mu_0$ has nonzero entries is essential for Theorem \ref{th:equivalent}. It is used in the proof. In particular the  equivalence between \eqref{intro:varphi_nu_mu_0} and \eqref{intro:varphi_nu_mu_0_fin} holds true only if one impose this condition.
	\end{remark}

	\begin{proof}[Proof of Theorem \ref{th:equivalent}]
		First, let us show that, if  $(m^*(\cdot),\varphi^*(\cdot),\mu^*(\cdot),\nu^*(\cdot))$ is the solution of control problem~(\ref{functional:J_control}),~(\ref{sys:control}) such that $J(m^*(\cdot),\varphi^*(\cdot),\mu^*(\cdot),\nu^*(\cdot))=0$, then $\nu^*(\cdot)$ provides the solution of the planning problem. 
		
		The equality $J(m^*(\cdot),\varphi^*(\cdot),\mu^*(\cdot),\nu^*(\cdot))=0$ means that
		\begin{equation} \label{equality:mu_varphi_star}
			\mu_0 \varphi^*(0) - \mu^*({T}) \varphi^*({T}) - \int_{0}^{T} \mu^*(t) {g}(t,m^*(t),\nu^*(t)) dt = 0.
		\end{equation}
		The first two terms in the left-hand side of this equality can be expressed as
		\[\begin{split}
			\mu_0 \varphi^*(0) - \mu^*({T}) \varphi^*({T})&=-\int_0^T\frac{d}{dt}[\mu^*(t)\varphi^*(t)]dt \\&= -\int_{0}^{T} \left[ \frac{d}{dt}{\mu}^*(t) \varphi^*(t) + \mu^*(t) \frac{d}{dt}{\varphi}^*(t) \right] dt.\end{split}
		\]
		Plugging this into~(\ref{equality:mu_varphi_star}) and using the second and third equations in~(\ref{sys:control}), we obtain
		\[
		\begin{split}
			-\int_{0}^{T} \mu^*(t)  {g}(t,&m^*(t),\nu^*(t)) dt \\=
			\int_{0}^{T} &\mu^*(t) \Qcal(t,m^*(t),\nu^*(t)) \varphi^*(t) dt -\int_{0}^{T} \mu^*(t) H(t,m^*(t),\varphi^*(t)).
		\end{split}
		\]
		Hence,
		\begin{equation}\label{equality:int_Q_g_H}
			\begin{split}
				\int_{0}^{T} \mu^*(t) [ \Qcal(t,m^*(t),\nu^*(t)) \varphi^*(t&) + {g}(t,m^*(t),\nu^*(t))\\&-H(t,m^*(t),\varphi^*(t)) ] dt = 0.
			\end{split}
		\end{equation}
		
		Notice that each component of the vector
		\begin{equation}\label{vector:Q_phi}
			[\Qcal(t,m^*(t),\nu^*(t)) \varphi^*(t) +{g}(t,m^*(t),\nu^*(t))-H(t,m^*(t),\varphi^*(t)) ]
		\end{equation}
		is nonpositive.
		Moreover, due to the assumption that $\mu_0$ has nonzero coordinates, we have that $\mu_i^*(t)>0$ for each $i=1,\ldots,d$ and every $t\in [0,T]$.
		This implies the inequality
		\[
		\mu^*(t) [ \Qcal(t,m^*(t),\nu^*(t)) \varphi^*(t) + {g}(t,m^*(t),\nu^*(t))-H(t,m^*(t),\varphi^*(t)) ] \leq 0.
		\]
		Therefore, due to equality (\ref{equality:int_Q_g_H}) and the nonpositiveness of the components of (\ref{vector:Q_phi}), we conclude that, for every $t\in [0,T]$,
		\[
		\mu^*(t) [ \Qcal(t,m^*(t),\nu^*(t)) \varphi^*(t) + {g}(t,m^*(t),\nu^*(t))-H(t,m^*(t),\varphi^*(t)) ]= 0.
		\]
		Using once more time the fact that $\mu^*(t)$ has strictly positive coordinates, we arrive at inclusion
		\[
		\nu^*(t)\in\ArgMax_{\nu\in (\mathcal{P}(U))^d}[ \Qcal(t,m^*(t),\nu) \varphi^*(t) + {g}(t,m^*(t),\nu) ].
		\]
		This, together with the first and the second equation in~(\ref{sys:control}) imply the fact that $\nu^*(\cdot)$ is the solution of the planning problem in the sense of Definition \ref{def:planning_classical}.
		
		The converse implication directly follows from Definition \ref{def:planning_classical}, equality \eqref{equality:varphi_mean_mu} and the fact that the distribution $\mu(t)=(\mu_1(\cdot),\ldots,\mu_d(\cdot))$ defined by the rule $\mu_k(t)=\mathbb{P}^{0,\mu_0}_{\nu(\cdot),m(\cdot)}\{\omega(t)=k\}$ satisfies \eqref{eq:Kolmogorov_linear_nu} for $s=0$.
	\end{proof}
	
	\section{Minimal regret solutions}\label{sect:regularization}
	The example presented in Section~\ref{section:example} shows that planning problem may not have a solution, even if $m_T$ lies in reachable set for $m_0$.
	To overcome this problem, we introduce the generalized solution of the planning problem (see Definition~\ref{def:planning_generalized}  below).
	
	First, let us recall the topology on the set of randomized strategies $\mathcal{U}$.
	We say that the sequence of randomized controls $\{\nu^n(\cdot)\}_{n=1}^\infty$, where for each $n$, \[\nu^n(\cdot)=(\nu^n_i(\cdot))_{i=1}^d,\] converges to $\nu^*(\cdot)=(\nu^*_i(\cdot))_{i=1}^d$ if, for every $i=1,\ldots,d$ and every $\zeta\in C([0,T]\times U)$, 
	\[
	\int_{0}^T\int_{U}\zeta(t,u)\nu^n_i(t,du)dt \rightarrow \int_{0}^T\int_{U}\zeta(t,u)\nu^*_i(t,du)dt \text{ as } n \rightarrow \infty.
	\]
	Alternatively,  this converges can be defined through the Young measures as follows.
	If  $\xi:[0,T]\rightarrow \mathcal{P}(U)$ is a weakly measurable function, then the corresponding Young measure $\lambda\star\xi$ is defined by the rule: for $\zeta\in C([0,T]\times U)$,
	\[
	\int_{[0,T] \times U}\zeta(t,u)(\lambda\star\xi)(d(t,u)) \triangleq \int_{0}^T\int_U\zeta(t,u)\xi(t,du)dt.
	\]
	Hereinafter, $\lambda$ stands for the Lebesgue measure on $[0,T]$.
	Thus, the convergence of sequence $\{\nu^n(\cdot)\}_{n=1}^\infty$ to $\nu^*(\cdot)$ is the narrow convergence of Young measures $\{\lambda\star\nu^n_i\}$ to $\{\lambda\star\nu^*_i\}$ for each $i\in \{1,\ldots,d\}$.
	Due to the Banach-Alaoglu theorem \cite[Theorem 6.21]{Infinite_dimensional_analysis}, the set of Young measures is compact. Therefore, the set of randomized control is also compact. 
	
	Finally, notice that the  convergence of the randomized strategies can be metricized.
	Indeed, choose a sequence of continuous functions defined on $[0,T]\times U$ with values in $\mathbb{R}$ $\{\psi_l\}_{l=1}^\infty$ to be a dense in $C([0,T]\times U)$.
	Without loss of generality, one can assume that $\|\psi_l\|=1$.
	If $\nu^1(\cdot)=(\nu^1_i(\cdot))_{i=1}^d$, $\nu^2(\cdot)=(\nu^2_i(\cdot))_{i=1}^d$ are two randomized strategies, then set
	\[\begin{split}
		\mathbf{d}(\nu^1,\nu^2)\triangleq  \sum_{i=1}^d\sum_{l=1}^\infty 2^{-l}\bigg\vert \int_{[0,T]\times U}\psi_l(&t,u)\nu^1_i(d(t,u))\\&-\int_{[0,T]\times U}\psi_l(t,u)\nu^2_i(d(t,u))\bigg\vert.\end{split}\]
	
	Since, by assumption, the sequence  $\{\psi_l\}$ is dense $C([0,T]\times U)$, we have that 
	\begin{enumerate}[label=(\roman*)]
		\item $\mathbf{d}$ is a metric on $\mathcal{U}$;
		\item the sequence $\{\nu^n\}_{n=1}^\infty$ converges to $\nu^*$ if and only if $\mathbf{d}(\nu^n,\nu^*)\rightarrow 0$ as $n\rightarrow\infty$.
	\end{enumerate}
	
	As in Section \ref{sect:equivalent}, let $\mu_0\in\Sigma^d$ be a given distribution with nonzero coordinates.
	
	\begin{definition}\label{def:planning_generalized}
		We say that $\gamma^*(\cdot)\in\mathcal{U}$ is a \textit{$\mu_0$-minimal regret solution} of the planning problem provided that, for each natural $n$, there exists a 5-tuple $(\alpha^n,m^n(\cdot),\varphi^n(\cdot),\mu^n(\cdot),\nu^n(\cdot))$ satisfying the following conditions:
		\begin{enumerate}[label=(\roman*)]
			\item\label{def:cond:four_tuple} the 4-tuple $(m^n(\cdot),\varphi^n(\cdot),\mu^n(\cdot),\nu^n(\cdot))$  is a solution of  optimal control problem~(\ref{functional:J_control}),~(\ref{sys:control}) with additional constraint $\|\varphi^n(T)\|\leq\alpha^n$;
			\item\label{def:cond:alpha} $\alpha^n \uparrow +\infty$ as $n\rightarrow\infty$;
			\item\label{def:cond:gamma} $\nu^n(\cdot) \to \gamma^*(\cdot)$ as $n\rightarrow\infty$.
		\end{enumerate}
		In the following, we say that the sequence $\{(\alpha^n,m^n(\cdot),\varphi^n(\cdot),\mu^n(\cdot),\nu^n(\cdot))\}_{n=1}^\infty$ determines  the solution~$\gamma^*(\cdot)$.
	\end{definition}
	The set of minimal regret solutions is denoted by $\gsol$.
	
	Let us briefly comment the concept of minimal regret solution. As we mentioned above, the original planning problem can be regraded as a two-level game. On the first level an external decision maker chooses a terminal payoff $\sigma$ wishing to stir the system of player to the final distribution $m_T$. On the second level, we have a noncooperative game for the infinite number of similar players. The example presented in Section \ref{section:example} shows the nonefficiency of the Nash equilibrium on the second level. It seems that, in this case,  an equilibrium appears if we allow unbounded terminal payoffs. Further, we notice (see  Theorem \ref{th:equivalent}) that the equilibrium in the second level game is equivalent to the zero regret averaged according to a distribution $\mu_0$. One can use here any element of $\Sigma^d$  with strictly positive entries.  The coordinates of the vector $\mu_0$ can be interpreted as preferences of  initial states. Certainly, from the players' point of view, the most natural choice is $\mu_0=m_0$. However, the external decision maker can use different preferences.
	
	The minimal regret solution implies that the external decision maker prescribes the players to use a strategy that is a cluster point of strategies those provide a minimization for weighted regret for bounded terminal payoffs and stir the whole system to the desired distribution $m_T$. The way of averaging is fixed and is given by a distribution $\mu_0$ that, as above, describes the preferences of the external decision maker. The only condition on $\mu_0$ is the strict positivity of its entries. The condition that  bounds on the norms of the terminal payoff tend to infinity implies that we, in fact, allow  unbounded terminal payoffs.  
	
	The minimal regret solution is in a good agreement with the original definition of the planning problem whenever the latter has a solution.  The main properties of the minimal regret solution are formulated in the following statements. 
	Here, $\mu_0$ is a fixed row-vector with nonzero coordinates. 
	
	\begin{theorem}\label{th:link_minimal_regret}
		The following properties hold true:
		\begin{enumerate}
			\item there exists at least one $\mu_0$-minimal regret solution of the planning problem;
			\item the set of $\mu_0$-minimal regret solutions of the planning problem is closed;
			\item if the set of a classical solutions is not empty, then its closure is equal to the set of  $\mu_0$-minimal regret solutions.
		\end{enumerate}
	\end{theorem}
	
	\begin{proof}
		As we mentioned above, the set of randomized strategies is compact.
		Moreover, the ball of the finite radius is also compact.
		Notice that the functions $m(\cdot)$, $\varphi(\cdot)$ and $\mu(\cdot)$ satisfy (\ref{sys:control}) and, thus, depend on $\nu(\cdot)$ and $\varphi(T)$ in the continuous way.
		This fact directly follows from the definition of convergence on the set of randomized strategies and the continuous dependence of the solution of differential equation on a parameter and an initial condition.
		Therefore, control problem (\ref{functional:J_control}), (\ref{sys:control}) with the additional constraint $\|\varphi(T)\|\leq\alpha$ has at least one solution.
		Now recall that the $\mu_0$-minimal regret solution is a cluster point of solutions of problems~(\ref{functional:J_control}),~(\ref{sys:control}) with the additional constraint $\|\varphi(T)\|\leq\alpha$ in the case when $\alpha\uparrow\infty$.
		Due to compactness of the  space of randomized strategies, one can find a sequence $\{(\alpha^n,m^n(\cdot),\varphi^n(\cdot),u^n(\cdot),\nu^n(\cdot))\}_{n=1}^\infty$ and a randomized strategy $\gamma^*(\cdot)\in\mathcal{U}$ such that each 4-tuple
		$(m^n(\cdot),\varphi^n(\cdot),u^n(\cdot),\nu^n(\cdot))$ solves optimization problem (\ref{functional:J_control}), (\ref{sys:control}) for $\|\varphi^n(T)\|\leq\alpha^n$, while $\alpha^n\rightarrow\infty$ and $\nu^n(\cdot)\rightarrow\gamma^*(\cdot)$.
		This proves the first statement of the theorem.
		
		To prove the second statement of the theorem, consider a randomized control $\gamma^*(\cdot)$ lying  in the closure of the set of minimal regret solutions $\gsol$.
		Then there exists a sequence $\{\gamma^k(\cdot)\}_{k=1}^\infty \subset \gsol$ that converges to $\gamma^*(\cdot)$.
		Due to Definition~\ref{def:planning_generalized}, for every  $\gamma^k(\cdot)$, one can find a sequence of 5-tuples $\{(\alpha^{k,n},m^{k,n}(\cdot),\varphi^{k,n}(\cdot),\mu^{k,n}(\cdot),\nu^{k,n}(\cdot))\}_{n=1}^\infty$ determining $\gamma^k(\cdot)$.
		In particular, for every $k$, $\alpha^{k,n}\uparrow+\infty$ and $\nu^{k,n}(\cdot) \to \gamma^k(\cdot)$ as $n\to\infty$.
		Given a natural $k$, let $n_k$ be such that
		\begin{itemize}
			\item $\mathbf{d}(\nu^{k,n_k}(\cdot), \gamma^k(\cdot)) \leq 2^{-k}$;
			\item $\alpha^{k,n_k}> \alpha^{k-1,n_{k-1}}$ and $\alpha^{k,n_k} \geq k$.
		\end{itemize}
		We have that the 5-tuple  $(\alpha^{k,n_k},m^{k,n_k}(\cdot),\varphi^{k,n_k}(\cdot),\mu^{k,n_k}(\cdot),\nu^{k,n_k}(\cdot))$ satisfies condition \ref{def:cond:four_tuple} of Definition~\ref{def:planning_generalized}.
		Further, $\alpha^{k,n_k}\uparrow +\infty$ as $k\rightarrow\infty$.
		Finally,
		\[
		\mathbf{d}(\nu^{k,n_k},\gamma^*) \leq \mathbf{d}(\nu^{k,n_k},\gamma^k) + \mathbf{d}(\gamma^{k},\gamma^*).
		\]
		Since $\mathbf{d}(\nu^{k,n_k},\gamma^k)\leq 2^{-k}$ and $\mathbf{d}(\gamma^{k},\gamma^*)\to 0$ as $k\to \infty$, we conclude that the sequence $\{\mathbf{d}(\nu^{k,n_k},\gamma^*)\}$ converges to zero.
		Thus, the $\{(\alpha^{k,n_k},m^{k,n_k}(\cdot),\varphi^{k,n_k}(\cdot),\mu^{k,n_k}(\cdot),\nu^{k,n_k}(\cdot))\}_{k=1}^\infty$ satisfies the conditions of Definition \ref{def:planning_generalized} and $\gamma^*$ is a $\mu_0$-minimal regret solution.
		
		Now let us prove the third claim of the theorem. Here we assume that the set of classical solutions $\csol$ is nonempty. 
		Notice that each classical solution is a $\mu_0$-minimal regret solution.
		Indeed, for any $\nu(\cdot)\in\csol$, due to Definition~\ref{def:planning_classical}, there exist a distribution $m(\cdot)$ and a value function~$\varphi(\cdot)$ those are admissible for $\nu(\cdot)$.
		Further, let $\mu(\cdot)$ be such that
		\begin{equation}\label{eq:mu_gen}
			\frac{d}{dt}\mu(t)=\mu(t)\Qcal(t,m(t),\nu(t)),\ \ \mu(t_0)=\mu_0.
		\end{equation}
		Hence, $\nu(\cdot)$ satisfies Definition~\ref{def:planning_generalized} with the determining sequence $\{(n\|\varphi(T)\|,m(\cdot),\varphi(\cdot),\mu(\cdot),\nu(\cdot))\}_{n=1}^\infty$.
		Therefore,
		\begin{equation}\label{inclusion:class_generalized}
			\csol\subset \gsol.
		\end{equation}
		
		To show that $\operatorname{cl}(\csol)=\gsol$, choose $\gamma(\cdot)\in\gsol$.
		Let $\{(\alpha^n,m^n(\cdot),\varphi^n(\cdot),\mu^n(\cdot),\nu^n(\cdot))\}_{n=1}^\infty$ be a corresponding  determining sequence.
		Since, by assumption, $\csol$ is nonempty, one can take $\nu^*(\cdot)\in\csol$.
		Let $m^*(\cdot)$ and $\varphi^*(\cdot)$ be the corresponding distribution  and the value function respectively.
		The flow of distributions $m^*(\cdot)$ determines the function $\mu^*(\cdot)$ that solves the initial value problem~(\ref{eq:mu_gen}) with $m(t)=m^*(t)$ and $\nu(t)=\nu^*(t)$.
		By Theorem \ref{th:equivalent},
		\[
		J(m^*(\cdot),\varphi^*(\cdot),\mu^*(\cdot),\nu^*(\cdot)) = 0.
		\]
		Denote $A \triangleq \|\varphi^*(T)\|$.
		Consider $N$ such that $\alpha^n \geq A$ for all $n>N$.
		Hence, due to~\eqref{equality:varphi_mean_mu}, we have
		\[
		\begin{split}
			0&\leq J(m^n(\cdot),\varphi^n(\cdot),\mu^n(\cdot),\nu^n(\cdot)) \\&= \min \{J(m(\cdot),\varphi(\cdot),\mu(\cdot),\nu(\cdot)):\\ &\hspace{70pt} (m(\cdot),\varphi(\cdot),\mu(\cdot),\nu(\cdot))\text{ satisfies }\eqref{sys:control}\text{ and }\|\varphi(T)\|\leq\alpha^n\} \\
			&\leq J(m^*(\cdot),\varphi^*(\cdot),\mu^*(\cdot),\nu^*(\cdot)) = 0.
		\end{split}
		\]
		This means that each randomized strategy $\nu^n(\cdot)$ is a classical solution of the planning problem for $n>N$ (see Theorem \ref{th:equivalent}).
		Since $\nu^n(\cdot) \to \gamma(\cdot)$ as $n\rightarrow\infty$, we have that $\gamma(\cdot)$ lies in closure of the $\csol$.
		This and (\ref{inclusion:class_generalized}) yields that
		\[
		\operatorname{cl}(\csol)=\gsol.
		\]
	\end{proof}

	\section{Uniqueness analysis}\label{sect:uniqueness}
	In this section, in addition to conditions \ref{cond:compact}--\ref{cond:reacability}, we assume that
	\begin{enumerate}[label=(U\arabic*)]
		\item transition rates $Q_{i,j}$ depend only on $t$ and $u$;
		\item $g(t,m,u)=g^0(t,u)+g^1(t,m)$;
		\item\label{cond:uni:inequalities} either 
		\begin{enumerate}[label=(\roman*), series=listUnique]
			\item\label{cond:uni:inequalities:strict_monotonicity} for every $m_1,m_2\in\Sigma^d$, $m_1\neq m_2$, $t\in [0,T]$, \begin{equation}\label{cond:uni:monotonicity} (m_1-m_2)(g^1(t,m_1)-g^1(t,m_2))<0,\end{equation}
		\end{enumerate} 
		or
		\begin{enumerate}[label=(\roman*), resume*=listUnique]
			\item\label{cond:uni:inequalities:strict_concavity} the functions 
			\[\rd\ni\phi\mapsto H^0_i(t,\phi)\triangleq \max_{u\in U}\Bigg[\sum_{j=1}^dQ_{i,j}(t,u)\phi_j+g^0(t,u)\Bigg]\] are strictly concave for each $t\in [0,T]$, $m\in\Sigma^d$, $i\in \{1,\ldots,d\}$, while, for every $m_1,m_2\in\Sigma^d$, 
			\begin{equation}\label{cond:uni:inequalities:nonstrict_Lasry_lions}
				(m_1-m_2)(g^1(t,m_1)-g^1(t,m_2))\leq 0
			\end{equation}
		\end{enumerate}
	\end{enumerate}
	
	Notice that conditions \eqref{cond:uni:monotonicity}, \eqref{cond:uni:inequalities:nonstrict_Lasry_lions} are  Lasry-Lions monotonicity condition for the running payoff in the finite state case.
	
	We will study the uniqueness if there exists a classical solution.
	
	\begin{proposition}\label{prop:uniqueness}
		Let $\nu^1$ and $\nu^2$ be two solutions of the planning problem in the sense of Definition \ref{def:planning_classical}, while the distributions $m^k(\cdot)$ and the value functions $\varphi^k(\cdot))$ correspond to $\nu^k$, $k=1,2$. Then, $m^1(\cdot)=m^2(\cdot)$, $\varphi^1(\cdot)=\varphi^2(\cdot)$. Moreover, if part~\ref{cond:uni:inequalities:strict_concavity} of condition~\ref{cond:uni:inequalities} is fulfilled, then $\nu^1(\cdot)$ and $\nu^2(\cdot)$ coincide and $\nu^1(\cdot)=\nu^2(\cdot)$ is determined by a measurable feedback strategy, i.e., $u(\cdot)=(u_i(\cdot))_{i=1}^d$.
	\end{proposition}
	\begin{proof} The proof is by the standard Lasry-Lions monotonicity arguments. Since $m^1(0)=m^2(0)=m_0$, $m^1(T)=m^2(T)=m_T$, we have that
		\[0=\int_0^T\frac{d}{dt}[(m^1(t)-m^2(t))(\varphi^1(t)-\varphi^2(t))]dt. 
		\] Using the fact that $(m^1(\cdot),\varphi^1(\cdot))$ (respectively, $(m^2(\cdot),\varphi^2(\cdot))$) corresponds to the classical solution of the planning problem $\nu^1(\cdot)$ (respectively, $\nu^2(\cdot)$), we conclude that 
		\[\begin{split}
			0=\int_0^T\bigg[(m^1(t)\mathcal{Q}(t,\nu^1(t))- m^2(t)&\mathcal{Q}(t,\nu^2(t)))(\varphi^1(t)-\varphi^2(t))\\+ (m^1(t)-m^2(t))(-\mathcal{Q}&(t,\nu^1(t))\varphi^1(t)-g^0(t,\nu^1(t))-g^1(t,m^1(t))\\+ \mathcal{Q}&(t,\nu^2(t))\varphi^2(t)+g^0(t,\nu^2(t))+g^1(t,m^2(t)))\bigg]dt. \end{split}\] Simple algebra gives that
		\begin{equation}\label{equality:zero_uniqueness}\begin{split}
				0=\int_0^T &\bigg[m^1(t)(\mathcal{Q}(t,\nu^2(t))\varphi^2(t)+ g^0(t,\nu^2(t))-\mathcal{Q}(t,\nu^1(t))\varphi^2(t)-g^0(t,\nu^1(t)))\\&+m^2(t)(\mathcal{Q}(t,\nu^1(t))\varphi^1(t)+g^0(t,\nu^1(t))-\mathcal{Q}(t,\nu^2(t))\varphi^1(t)-g^0(t,\nu^2(t)))\\ &+
				(m^1(t)-m^2(t))(g^1(t,m^2(t))-g^1(t,m^1(t)))\bigg]dt.
			\end{split}
		\end{equation}
		
		If part \ref{cond:uni:inequalities:strict_monotonicity} of condition \ref{cond:uni:inequalities} is fulfilled, we have that 
		\begin{itemize}
			\item $\mathcal{Q}(t,\nu^2(t))\varphi^2(t)+g^0(t,\nu^2(t))-\mathcal{Q}(t,\nu^1(t))\varphi^2(t)-g^0(t,\nu^1(t))\geq 0$;
			\item $\mathcal{Q}(t,\nu^1(t))\varphi^1(t)+g^0(t,\nu^1(t))-\mathcal{Q}(t,\nu^2(t))\varphi^1(t)-g^0(t,\nu^2(t))\geq 0$.
		\end{itemize} Moreover, if, at some point $t\in [0,T]$, $m^1(t)\neq m^2(t)$, then at this time instant one has that 
		\[(m^1(t)-m^2(t))(g^1(t,m^2(t))-g^1(t,m^1(t)))>0.\] Therefore, we have that equality \eqref{equality:zero_uniqueness} holds true only if $m^1(t)\equiv m^2(t)$. This implies the equality $\varphi^1(t)\equiv\varphi^2(t)$.
		
		Now let us consider the case where part \ref{cond:uni:inequalities:strict_concavity} of condition \ref{cond:uni:inequalities} is valid.
		
		Notice that if $\varphi^1(\cdot)\neq \varphi^2(\cdot)$, then the first two terms in the right-hand side of \eqref{equality:zero_uniqueness} are strictly positive, whereas the third term is nonnegative. This contradicts with equality \eqref{equality:zero_uniqueness}. Therefore, we have that $\varphi^1(\cdot)=\varphi^2(\cdot)$. This and the strict concavity of the function $H^0(t,\phi)$ give that 
		$\nu^1(t)=\nu^2(t)$ a.e. Since each function $m^i(\cdot)$ satisfies the Kolmogorov equation
		\[\frac{d}{dt}m^i(t)=\mathcal{Q}(t,\nu^i(t)),\ \ m^i(0)=m_0,\] we obtain the equality $m^1(t)=m^2(t)$, $t\in [0,T]$. 
		
		Finally, notice that, if the functions $H^0_i(t,\phi)$ are strictly concave w.r.t. $\phi$, then, for every $m(\cdot)$ and $\varphi(\cdot)$, there exists a unique $u(\cdot)$ such that
		\[u^\sharp(t)\in \ArgMax\limits_{u \in U^d} \left\{ \Qcal(t,u) \varphi(t) + g^0(t,u)\right\}.\] Therefore, in this case, $\nu^1(\cdot)=\nu^2(\cdot)$ is determined by the strategy $u^\sharp(\cdot)$.
	\end{proof}
	
	\begin{corollary}
		Assume that there exists at least one classical solution, whereas $\mu_0\in\Sigma^d$ has strictly positive coordinates. 
		
		If part \ref{cond:uni:inequalities:strict_monotonicity} of condition \ref{cond:uni:inequalities} is fulfilled, then $\gsol=\csol$ and there exists a unique pair $(m(\cdot),\varphi(\cdot))$ that corresponds to every $\nu(\cdot)\in\csol$.
		
		If part \ref{cond:uni:inequalities:strict_concavity} of condition \ref{cond:uni:inequalities} holds, then 
		$\gsol=\csol=\{\nu(\cdot)\}$, where, for each $t\in [0,T]$, $\nu(t)=(\delta_{u_i(t)})_{i=1}^d$.
	\end{corollary}
	
	The corollary directly follows from Proposition \ref{prop:uniqueness} and the third statement of Theorem \ref{th:link_minimal_regret}.
	
	\section{Conclusion}
	Previous works showed the existence of solution of the planning problem for the case of second order~\cite{Bakaryan_Ferrira_Gomes_2,Bakaryan_Ferrira_Gomes,lions_lecture, Porrtte} and first order~\cite{Graber_et_al_2019, var_approach_mfg} mean field games.
	All aforementioned papers assumed that the Hamiltonian has superlinear growth w.r.t. the adjoint variable.
	We considered the finite state mean field game with sublinear Hamiltonian and demonstrated that, in this case, the planning problem may has no solution. Apparently, analogous examples can be found for the second- and first-order mean field games with sublinear Hamiltonian.  
	
	Additionally, we introduced the concept of minimal regret solution that is a generalized solution for the planning problem obtained by replacing the condition that the representative player maximize his/her payoff for some terminal payoff by the claim that the solution is a limit of solutions of certain optimal control problems defined on compact spaces.
	We proved the existence of the minimal regret solution and the fact that the set of minimal regret solution is the  closure of the set of classical solutions of planning problem if the latter is nonempty. The uniqueness is proved for the case where there exists a classical solution of the planning problem under conditions those include the Lasry-Lions 
	
	Notice that, in the case where there is no classical solution of the planning problem, the minimal regret solution can principally depend on the choice of the initial distribution of the representative player (see Definition~\ref{def:planning_generalized}).
	The study of this dependence is the theme of future research. Moreover, the future works will include more general uniqueness conditions as well as the extension of the concept of minimal regret solution to the case of first- and second-order mean field games with sublinear Hamiltonian.
	
	\textbf{Acknowledgments.} \quad
	The authors would like to thank the anonymous reviewers for their valuable comments and suggestions. 
	
	The work of Yurii Averboukh on this paper was in the framework of a research grant funded by the Ministry of Science and Higher Education of the Russian Federation (grant ID: 075-15-2020-928).
	
	\bibliography{PPMarkovMFG}
\end{document}